\documentclass[12pt]{amsart}

\usepackage{amsfonts}
\usepackage{amssymb}
\usepackage{graphicx}
\usepackage{amsmath}
\usepackage{latexsym}
\usepackage{amscd}
\usepackage{xypic}
\usepackage[all,cmtip]{xy}
\usepackage{mathrsfs}
\usepackage{enumitem} 
\usepackage{braket}
\usepackage[margin=1.5in]{geometry}
\usepackage{esint}
\usepackage{dsfont}
\usepackage{pgf}
\usepackage{xypic,tikz}

\usepackage{hyperref}

\allowdisplaybreaks

\newtheorem{prop}{Proposition}
\newtheorem{theo}[prop]{Theorem}

\theoremstyle{definition}

\newtheorem{rema}[prop]{Remark}

\newtheorem{defi}[prop]{Definition}

\newcommand{\NN}{\mathbb{N}}

\renewcommand{\SS}{\mathbb{S}}

\newcommand{\cB}{\mathcal B}

\renewcommand{\cL}{\mathcal L}

\DeclareMathOperator{\area}{area}

\define{\sff}{\text{\rm II}}
\define{\tfsff}{\accentset{\circ}{\sff}}

\let\oldmarginpar\marginpar
\renewcommand\marginpar[1]{\-\oldmarginpar[\raggedleft\footnotesize #1]%
{\raggedright\footnotesize #1}}

\DeclareMathOperator{\Index}{index}

\DeclareMathOperator{\genus}{genus}
\usepackage{graphicx}

\allowdisplaybreaks

\title{A note on minimal surfaces with bounded index}

\address{Department of Mathematics, University of Pennsylvania}

\author{Davi Maximo}

\date{August 24, 2018}

\begin{document}
\begin{abstract}
For any closed Riemannian three-manifold, we prove that for any sequence of closed embedded minimal surfaces with uniformly bounded index, the genus can only grow at most linearly with respect to the area.
\end{abstract}

\maketitle

\section{Introduction}

There has been recently spectacular progress on the existence theory of minimal hypersurfaces led by the landmark work of F. Marques and A. Neves on the Almgren-Pitts minmax theory \cite{MaNe14, MaNe17, MaNe16, MaNe18} and the volume spectrum \cite{LiMaNe18} (jointly with Y. Liokumovich). For instance, by Marques-Neves \cite{MaNe17} and A. Song \cite{Song18}, we now know that any closed Riemannian manifold of dimension $3\leq n\leq 7$ contains {\it infinitely many} smoothly embedded closed minimal hypersurfaces. In addition, when the metric on $M$ is {\it generic}, Irie-Marques-Neves \cite{IrMaNe18} showed that the union of embedded minimal hypersurfaces is {\it dense}. Soon after,  Marques-Neves-Song \cite{MaNeSo17} proved that there is a sequence of such minimal hypersurfaces that is actually {\it equidistributed}. Finally, in 3 dimensions and again for generic metrics,  Chodosh-Mantoulidis \cite{CM18} showed the existence of minimal surfaces with arbitrarily large area, genus, and Morse index.

It is interesting to ask in general how the index, topology, and geometry (area and curvature) of minimal surfaces relate to each other. The answer can be subtle. For instance, in presence of positive Ricci curvature, Choi--Schoen \cite{ChSc85} proved that a sequence of minimal surfaces with bounded genus must have bounded area, curvature, and index. In a general Riemannian three-manifolds, however, it is not possible to bound the area nor the index of an embedded minimal surface by the genus, even if one assumes positive scalar curvature, as it can be seen in examples constructed by Colding--De Lellis \cite{CoDe05}. On the other hand, jointly with O Chodosh and D. Ketover \cite{CKM17}, we were able to show that manifolds positive scalar curvature, uniform index bounds do imply uniform area and genus bounds. However, without ambient curvature assumptions, there can be sequences of minimal surfaces with uniformly bounded index but arbitrarily large genus and area \cite{CKM17}.
 
The purpose of this note is to prove the following estimate:
\begin{theo}\label{thm:main}
Let $(M^3,g)$ be a fixed closed Riemannian three-manifold. For any natural number $I$, there exists a constant $C=C(M,g,I)>0$ such that if $\Sigma_j$ is a sequence of two-sided, closed, connected, embedded minimal surfaces with $\Index(\Sigma_j)\leq I$, then
  $$\genus(\Sigma_j)<C \area(\Sigma_j).$$ 
\end{theo}

The proof of Theorem \ref{thm:main} relies on the local structure and surgery theorems proven in \cite{CKM17} which give a complete picture of how a sequence with uniformly bounded index can degenerate. Tools from there-manifold topology, namely methods of laminations and branched surfaces, are also key.

\begin{rema}
	The proof of Theorem \ref{thm:main} can be modified by routine arguments, as in \cite{CKM17}, so that the result remains true for one-sided minimal surfaces.
\end{rema}

\begin{rema}
The estimate in Theorem \ref{thm:main} is essentially sharp as seen by the following example. In \cite{Ja70}, Jaco constructed a sequence of incompressible surfaces $S_k$ in $S \times \SS^{1}$, where $S$ is a closed oriented surface of genus $r>1$. These surfaces have the property that, for each $k$, the projection $p:S \times \SS^{1} \rightarrow S$ restricted to $S_k$ is a covering map of $S$ of order $k$. In particular, their Euler characteristic satisfy $\chi(S_k)=k\chi(S)$, so the genus of $S_k$ grows linearly in $k$.

Fixing a constant curvature metric in $S$, we may then minimize area  in each homotopy class of $\Sigma_k$ in $S \times \SS^{1}$ using \cite{ScYa79} and obtain as a result a stable minimal surface $\Sigma_k$ which is embedded (after passing to a one-sided quotient, if necessary) by \cite{FHS83}.  These will have uniformly bounded curvature in $k$, by \cite{Sch83}, and thus area growth at least linear in $k$ since $p|_{\Sigma_k}: \Sigma_k \rightarrow S$ is a $k$-cover.
\end{rema}

\begin{rema}
Colding--Minicozzi \cite{CoMi00} showed that there is a $C^{2}$-open set of metrics $g$ on any manifold $M$ so that there is a sequence of embedded stable minimal tori $\Sigma_{j}$ with $\area_{g}(\Sigma_{j})\to\infty$ and thus the reverse inequality in Theorem \ref{thm:main} does not hold in general. 
\end{rema}

\begin{rema}
The estimate in Theorem \ref{thm:main} seems to be new in the literature even in the particular case of stable minimal surfaces.  
\end{rema}

\begin{rema}
	H. Rosenberg and B. Meeks have told me they have a different approach to obtain a similar bound to Theorem \ref{thm:main}.
\end{rema}

\subsection{Acknowledgements} I would like to thank H. Rosenberg for asking me whether such bound were true. I would also like to thank  O. Chodosh and J. Nogueira for interesting discussions about laminations.

\section{Preliminaries}

\subsection{Minimal laminations} Minimal laminations are a natural generalization of minimal surfaces. They are very convenient in dealing with a sequence of surfaces without area bounds. We recall their definition:

\begin{defi}
A closed set $\cL$ of $M^3$ is called a \textit{minimal lamination} if it is the union of pairwise disjoint, connected, injectively immersed minimal surfaces, which we call \textit{leaves}. For each point $x\in \cL$, we ask for the existence of a neighborhood $\Omega$ of $x$ and a $C^{0,\alpha}$ local coordinate chart $\Phi:\Omega \rightarrow \mathbb{R}^3$ under which image the leaves of $\cL$ pass through in slices of the form $\mathbb{R}^2\times\{t\} \cap \Phi(\Omega)$.
\end{defi}

\begin{rema}\label{rem:stable}
As an illustration of how laminations naturally appear when dealing with embedded minimal surfaces with no area bounds, consider a sequence of closed stable minimal surfaces $\Sigma_j$ on $(M^3,g)$. Suppose their areas are blowing up, $i.e.,$ $|\Sigma_j|\rightarrow\infty$\footnote{Such sequence of stable minimal surfaces with unbounded area appear, for example, in \textit{every metric} of the 3-torus, see Example 1.13 of \cite{CKM17}}. By curvature estimates of Schoen \cite{Sch83}, $\Sigma_j$ must have uniformly bounded second fundamental form, that is: 
$$|\textrm{II}_{\Sigma_j}|\leq C$$
for some $C>0$. Thus, after a subsequence, $\Sigma_j$ converges locally smoothly to a minimal lamination $\cL$ (see, for instance, page 475 by \cite{MeRo06}). 
\end{rema}

\subsection{Branched surfaces}\label{sec:branched} To analyze laminations that might appear as a sequence of minimal surfaces, we will consider branched surfaces that carry them. We remark that such methods were pioneered by Williams \cite{Wi74} and have been successfully used in geometry topology on several occasions, $e.g.$ see \cite{Li06,Li07} and, more recently, \cite{CoGa18}. They are higher dimensional analogs of train tracks carrying geodesic laminations, see \cite{PeHa92}.

\begin{defi}
A {\it branched surface} $B$ on a three-manifold $M$ is a two-complex consisting of a finite union surfaces that combine along the 1-skeleton giving it a well-defined tangent space at every point, and generic singularities. Every point $p \in B$ thus has a neighborhood in
$M$ which is homeomorphic to one of the local models in Figure \ref{fig:branched}:
\end{defi}
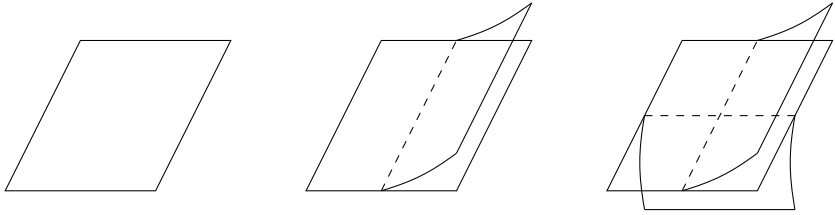
\begin{figure}[h]
	\begin{tikzpicture}
\draw (-7,-1) -- (-6,1) -- (-4,1) -- (-5,-1) -- cycle;

\draw (-3,-1) -- (-2,1) -- (0,1) -- (-1,-1) -- cycle;
	\draw (-1,-0.5) -- (0,1.5);
	\draw (-1,1) to [bend right =10] (0,1.5);
	\draw (-2,-1)  to [bend right =10]  (-1,-0.5);
	\draw[dashed] (-2,-1) -- (-1,1);

\draw (1,-1) -- (2,1) -- (4,1) -- (3,-1) -- cycle;
\draw (3,-0.5) -- (4,1.5);
\draw (3,1) to [bend right =10] (4,1.5);
\draw (2,-1)  to [bend right =10]  (3,-0.5);
\draw[dashed] (2,-1) -- (3,1);
\draw (1.5,-1.25) -- (3.5,-1.25);
\draw[dashed] (1.5,0) -- (3.5,0);
\draw (1.5,0) to [bend right =10] (1.5,-1.25);
\draw (3.5,0) to [bend right =10] (3.5,-1.25);
\end{tikzpicture}	\caption{Local models for a branched surface}\label{fig:branched}
\end{figure}
For branched surface $B$, we define its \textit{branched locus} to be the set of points in $B$ which do not have a neighborhood diffeomorphic to $\mathbb{R}^2$. We call the closure of each component of $B\setminus L$ a \textit{branch sector}. As with train tracks, a branched surface has a well-defined normal bundle in a 3-manifold and for any $\varepsilon>0$ sufficiently small, an $\varepsilon$-tubular neighborhood $N_\varepsilon(B)$ can be foliated by intervals transverse to $B$ as an $I$-bundle in such a way that collapsing these intervals collapses $N(B)$ to a new branched surface which can be canonically identified with $B$. 

We say that a surface $\Sigma$ (or, similarly, a lamination $\cL$) is \textit{fully carried} by $B$ if there exists $\varepsilon>0$ such that $S \subset N_\varepsilon(B)$ transversely intersects every $I$-fiber of $N_\varepsilon(B)$. If $\pi:N(B)\rightarrow B$ is the projection that collapses every $I$-fiber to a point and $b_1,\ldots,b_N$ are the components of $B\setminus L$, we let $x_i= |\Sigma \cap \pi^{-1}(b_i)|$ for each $b_i$. We can describe $\Sigma$ combinatorially via a non-negative integer coordinate $(x_1,\ldots,x_N)\in\mathbb{R}^N$ satisfying an obvious system of branch equations coming the from intersections of the respective branch sectors , see \cite{FlOe84,Oe88}:

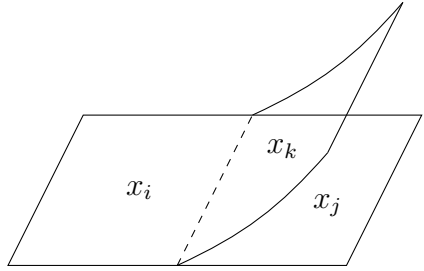
\begin{figure}[h]
	\begin{tikzpicture}

	\draw (-0.25,-1) -- (0.75,1) -- (5.25,1) -- (4.25,-1) -- cycle;
	\draw (4,0.5) -- (5,2.5);
	\draw (3,1) to [bend right =12] (5,2.5);
	\draw (2,-1)  to [bend right =12]  (4,0.5);
	\draw[dashed] (2,-1) -- (3,1);
	
	\node at (1.5,0) {$x_i$};
	\node at (4,-.2) {$x_j$};	
	\node at (3.4,0.55) {$x_k$};
	
\end{tikzpicture}\caption{Branch equations $x_i=x_j+x_k$}\label{fig:brancheq}
\end{figure}

\begin{prop}\label{prop:lam}
Every lamination $\cL$ in $M$ is fully carried by a branched surface $B$.
\end{prop}
\begin{proof}
This follows by general topological arguments, see \cite{GaOe89,Ha80s}. However, we sketch an argument for the laminations relevant to the proof of Theorem \ref{thm:main}.

Suppose $\cL$ is a minimal lamination which is the limit of a sequence of closed embedded minimal surfaces $\Sigma_j$ with $|\textrm{II}_{\Sigma_j}|\leq C$ for some fixed $C>0$. Following a standard argument on page 475 of \cite{MeRo06}, we may find $r>0$ sufficiently small and cover $\cL$ with extrinsic balls $B_r(x_1),B_r(x_2),\ldots,B_r(x_k)$, where $x_1,x_2\ldots,x_k$  belong to $\cL$, such that the intersection of any leaf with any of the balls $B_r(x_i)$ is either empty or the graph of a function with small gradient over a subset of a disk passing through $x_i$. Such discs can be the isotoped and glued to form the desired branched surface.
\end{proof}	
\begin{rema}\label{rem:branched}
We remark that the argument above implies the constructed branched surface $B$ also carry the surfaces $\Sigma_j$ or $j$ sufficiently large. In fact, $\Sigma_j \subset N_{2r}(B)$.
\end{rema}

\subsection{Local structure and surgery theorems} In Remark \ref{rem:stable}, we saw that a sequence of stable minimal surfaces converges, locally smoothly and after passing to a subsequence, to a minimal lamination. In this section we recall the main results of \cite{CKM17}, which guarantees that closed embedded minimal surfaces with uniformly bounded index behave similarly qualitatively, up to controlled errors. What follows comprises of combining Theorem 1.17 and Corollary 1.19 of \cite{CKM17}:

\begin{theo}\label{thm:ckm}
There exist functions $\tilde r(I)$ and $\tilde m(I)$ with the following property. Fix a closed three-manifold $(M^{3},g)$ and a natural number $I \in \NN$. Suppose $\Sigma_{j}\subset (M,g)$ is a sequence of closed embedded minimal surfaces with $\Index(\Sigma_{j})\leq I$. Then, after passing to a subsequence, there is $C>0$ and a finite set of points $\cB_{j}\subset \Sigma_{j}$ with cardinality $|\cB_{j}|\leq I$ so that the curvature of $\Sigma_{j}$ is uniformly bounded away from the set $\cB_{j}$, i.e.,
\[
|\sff_{\Sigma_{j}}|(x) \min\{1,d_{g}(x,\cB_{j})\}\leq C,
\]
but not at $\cB_{j}$, i.e.,
\[
\liminf_{j\to\infty} \min_{p\in\cB_{j}}|\sff_{\Sigma_{j}}|(p) = \infty.
\]
Passing to a further subsequence, the points $\cB_{j}$ converge to a set of points $\cB_{\infty}$ and the surfaces $\Sigma_{j}$ converge locally smoothly, away from $\cB_{\infty}$, to some lamination $\cL \subset M \setminus \cB_{\infty}$. The lamination has removable singularities, i.e., there is a smooth lamination $\widetilde\cL\subset M$ so that $\cL = \widetilde\cL\setminus\cB_{\infty}$. Moreover, there exists $\varepsilon_{0}>0$ smaller than the injectivity radius of $(M,g)$ so that $\cB_{\infty}$ is $4\varepsilon_{0}$-separated and for any $\varepsilon \in (0,\varepsilon_{0}]$, taking $j$ sufficiently large guarantees that there exists embedded surfaces $\widetilde \Sigma_{j}\subset (M^{3},g)$ satisfying:
\begin{enumerate}[itemsep=5pt, topsep=5pt]
	\item The new surfaces $\widetilde\Sigma_{j}$ agree with $\Sigma_{j}$ outside of $B_{\varepsilon}(\cB_{\infty})$. 
	\item The components of $\Sigma_{j}\cap B_{\varepsilon}(\cB_{\infty})$ that do not intersect the spheres $ \partial B_{\varepsilon}(\cB_{\infty})$ transversely and the components that are topological disks appear in $\widetilde\Sigma_{j}$ without any change.
	\item The curvature of $\widetilde\Sigma_{j}$ is uniformly bounded, i.e.
	\[
	\limsup_{j\to\infty}\sup_{x\in\widetilde\Sigma_{j}}|\sff_{\widetilde\Sigma_{j}}|(x) <\infty.
	\]
	\item Each component of $\widetilde\Sigma_{j}\cap  B_{\varepsilon}(\cB_{\infty})$ which is not a component of $\Sigma_{j}\cap B_{\varepsilon}(\cB_{\infty})$ is a topological disk of area at most $2\pi\varepsilon^{2}(1+o(\varepsilon))$.
	\item The genus drops in controlled manner, i.e.,
	\[
	\genus(\Sigma_{j})-\tilde r(I) \leq \genus(\widetilde\Sigma_{j}) \leq \genus(\Sigma_{j}).
	\]
	\item The number of connected components increases in a controlled manner, i.e.,
	\[
	|\pi_{0}(\Sigma_{j})|\leq |\pi_{0}(\widetilde\Sigma_{j})| \leq |\pi_{0}(\Sigma_{j})| + \tilde m(I).
	\]
	\item While $\widetilde\Sigma_{j}$ is not necessarily minimal, it is asymptotically minimal in the sense that $\lim_{j\to\infty} \Vert H_{\widetilde\Sigma_{j}}\Vert_{L^{\infty}(\widetilde\Sigma_{j})} = 0$.
\end{enumerate}
Finally, the new surfaces $\widetilde\Sigma_{j}$ converge locally smoothly to the smooth minimal lamination $\widetilde \cL$.
\end{theo} 

\section{End of proof of Theorem \ref{thm:main}}

We argue by contradiction: if the desired constant does not exist, we may find a sequence $\Sigma_j$ of closed, two-sided, embedded minimal surfaces with $\Index(\Sigma_j)\leq I$ and such that:\
\begin{equation}\label{eqproof}
\frac{\genus(\Sigma_j)}{\area(\Sigma_j)} \nearrow \infty.
\end{equation}
By Theorem 1.1 of \cite{CKM17}, this implies $\area(\Sigma_j)\rightarrow\infty$; otherwise, bounded index and bounded area would imply bounded genus and thus contradict \eqref{eqproof}.

Passing to a subsequence if necessary, we may then apply Theorem \ref{thm:ckm} to $\Sigma_j$ and obtain, after surgery, a sequence of nearly minimal surfaces $\widetilde\Sigma_{j}$ which, by $(4), (5), (6)$, we may also assume are connected and  satisfy:
\begin{equation}\label{eqprooftil}
\frac{\genus(\widetilde\Sigma_j)}{\area(\widetilde\Sigma_j)} \nearrow \infty.
\end{equation}
This crucially uses the fact that the surgery procedure deletes at most $I$ components and replaces them with disks of comparable area.

Moreover, again by Theorem \ref{thm:ckm}, $\widetilde\Sigma_j$ converges locally smoothly to a smooth minimal lamination $\widetilde{\cL}$. By Proposition \ref{prop:lam}, we can find a branched surface $B$ that fully carries $\widetilde{\cL}$. As in Section \ref{sec:branched}, let $L$ be the branched locus of $B$ and $b_1,b_2,\ldots,b_N$ its branched sections, $i.e.$, the components of $B\setminus L$. Each $\Sigma_j$ then correspond to an non-negative integer coordinate:
$$\Sigma_j \longrightarrow (x^j_1,x^j_2\ldots,x^j_N)\in\mathbb{R}^N.$$
Moreover, by Remark \ref{rem:branched}, $\Sigma_j \subset N_{2r}(B)$ for $j$ sufficiently large, and it may be reconstructed from disks which are graphical over the sectors $b_1,\ldots,b_N$. Thus, 
$$x^j_1|b_1|+ x^j_2|b_2| +\cdots +x^j_N|b_N| = {O}(|\Sigma_j|),$$
where $|b_i|$ is the area of the branch sector $b_i$. 

On other hand, the Euler characteristic of can also be estimated using a triangulation obtained by gluing the graphical pieces over the sectors $b_1,\ldots,b_N$. This yields the estimate
$$ |\chi(\Sigma_j)| = {O}(x_1+x_2+\ldots+x_n)$$
which contradicts equation \eqref{eqprooftil} and we are done.

\bibliography{bib} 

\providecommand{\bysame}{\leavevmode\hbox to3em{\hrulefill}\thinspace}
\providecommand{\MR}{\relax\ifhmode\unskip\space\fi MR }
\providecommand{\MRhref}[2]{%
  \href{http://www.ams.org/mathscinet-getitem?mr=#1}{#2}
}
\providecommand{\href}[2]{#2}
\begin{thebibliography}{CKM17}

\bibitem[CDL05]{CoDe05}
Tobias~H. Colding and Camillo De~Lellis, \emph{Singular limit laminations,
  {M}orse index, and positive scalar curvature}, Topology \textbf{44} (2005),
  no.~1, 25--45. \MR{2103999}

\bibitem[CG18]{CoGa18}
Tobias~Holck Colding and David Gabai, \emph{Effective finiteness of irreducible
  {H}eegaard splittings of non-{H}aken {$3$}-manifolds}, Duke Math. J.
  \textbf{167} (2018), no.~15, 2793--2832. \MR{3865652}

\bibitem[CKM17]{CKM17}
Otis Chodosh, Daniel Ketover, and Davi Maximo, \emph{Minimal hypersurfaces with
  bounded index}, Inventiones mathematicae \textbf{209} (2017), no.~3,
  617--664.

\bibitem[CM00]{CoMi00}
Tobias~H. Colding and William~P. Minicozzi, II, \emph{Embedded minimal surfaces
  without area bounds in 3-manifolds}, Geometry and topology: {A}arhus (1998),
  Contemp. Math., vol. 258, Amer. Math. Soc., Providence, RI, 2000,
  pp.~107--120. \MR{1778099}

\bibitem[CM18]{CM18}
Otis Chodosh and Christos Mantoulidis, \emph{Minimal surfaces and the
  allen-cahn equation on 3-manifolds: index, multiplicity, and curvature
  estimates}.

\bibitem[CS85]{ChSc85}
Hyeong~In Choi and Richard Schoen, \emph{The space of minimal embeddings of a
  surface into a three-dimensional manifold of positive {R}icci curvature},
  Invent. Math. \textbf{81} (1985), no.~3, 387--394. \MR{807063}

\bibitem[FHS83]{FHS83}
Michael Freedman, Joel Hass, and Peter Scott, \emph{Least area incompressible
  surfaces in {$3$}-manifolds}, Invent. Math. \textbf{71} (1983), no.~3,
  609--642. \MR{695910}

\bibitem[FO84]{FlOe84}
W.~Floyd and U.~Oertel, \emph{Incompressible surfaces via branched surfaces},
  Topology \textbf{23} (1984), no.~1, 117--125. \MR{721458}

\bibitem[GO89]{GaOe89}
David Gabai and Ulrich Oertel, \emph{Essential laminations in {$3$}-manifolds},
  Ann. of Math. (2) \textbf{130} (1989), no.~1, 41--73. \MR{1005607}

\bibitem[Hat]{Ha80s}
Allen Hatcher, \emph{Measured lamination spaces for 3-manifolds}, Unpublished
  manuscript. http://pi.math.cornell.edu/~hatcher/Papers/ML.pdf.

\bibitem[IMN18]{IrMaNe18}
Kei Irie, Fernando~C. Marques, and Andr\'{e} Neves, \emph{Density of minimal
  hypersurfaces for generic metrics}, Ann. of Math. (2) \textbf{187} (2018),
  no.~3, 963--972. \MR{3779962}

\bibitem[Jac70]{Ja70}
William Jaco, \emph{Surfaces embedded in {$M^{2}\times S^{1}$}}, Canad. J.
  Math. \textbf{22} (1970), 553--568. \MR{0267596}

\bibitem[Li06]{Li06}
Tao Li, \emph{Heegaard surfaces and measured laminations. {II}. {N}on-{H}aken
  3-manifolds}, J. Amer. Math. Soc. \textbf{19} (2006), no.~3, 625--657.
  \MR{2220101}

\bibitem[Li07]{Li07}
\bysame, \emph{Heegaard surfaces and measured laminations. {I}. {T}he
  {W}aldhausen conjecture}, Invent. Math. \textbf{167} (2007), no.~1, 135--177.
  \MR{2264807}

\bibitem[LMN18]{LiMaNe18}
Yevgeny Liokumovich, Fernando~C. Marques, and Andr\'{e} Neves, \emph{Weyl law
  for the volume spectrum}, Ann. of Math. (2) \textbf{187} (2018), no.~3,
  933--961. \MR{3779961}

\bibitem[MN14]{MaNe14}
Fernando~C. Marques and Andr\'{e} Neves, \emph{Min-max theory and the
  {W}illmore conjecture}, Ann. of Math. (2) \textbf{179} (2014), no.~2,
  683--782. \MR{3152944}

\bibitem[MN16]{MaNe16}
\bysame, \emph{Morse index and multiplicity of min-max minimal hypersurfaces},
  Camb. J. Math. \textbf{4} (2016), no.~4, 463--511. \MR{3572636}

\bibitem[MN17]{MaNe17}
Fernando~C. Marques and Andr{\'e} Neves, \emph{Existence of infinitely many
  minimal hypersurfaces in positive {R}icci curvature}, Invent. Math.
  \textbf{209} (2017), no.~2, 577--616. \MR{3674223}

\bibitem[MN18]{MaNe18}
Fernando~C. Marques and Andr{\'e} Neves, \emph{Morse index of multiplicity one
  min-max minimal hypersurfaces}, arXiv preprint arXiv:1803.04273 (2018).

\bibitem[MNS17]{MaNeSo17}
Fernando~C. Marques, Andr{\'e} Neves, and Antoine Song, \emph{Equidistribution
  of minimal hypersurfaces for generic metrics}, arXiv preprint
  arXiv:1712.06238 (2017).

\bibitem[MR06]{MeRo06}
William~H. Meeks, III and Harold Rosenberg, \emph{The minimal lamination
  closure theorem}, Duke Math. J. \textbf{133} (2006), no.~3, 467--497.
  \MR{2228460}

\bibitem[Oer88]{Oe88}
Ulrich Oertel, \emph{Measured laminations in {$3$}-manifolds}, Trans. Amer.
  Math. Soc. \textbf{305} (1988), no.~2, 531--573. \MR{924769}

\bibitem[PH92]{PeHa92}
R.~C. Penner and J.~L. Harer, \emph{Combinatorics of train tracks}, Annals of
  Mathematics Studies, vol. 125, Princeton University Press, Princeton, NJ,
  1992. \MR{1144770}

\bibitem[Sch83]{Sch83}
Richard Schoen, \emph{Estimates for stable minimal surfaces in
  three-dimensional manifolds}, Seminar on minimal submanifolds, Ann. of Math.
  Stud., vol. 103, Princeton Univ. Press, Princeton, NJ, 1983, pp.~111--126.
  \MR{795231 (86j:53094)}

\bibitem[Son18]{Song18}
Antoine Song, \emph{Existence of infinitely many minimal hypersurfaces in
  closed manifolds}, Preprint (2018).

\bibitem[SY79]{ScYa79}
R.~Schoen and Shing~Tung Yau, \emph{Existence of incompressible minimal
  surfaces and the topology of three-dimensional manifolds with nonnegative
  scalar curvature}, Ann. of Math. (2) \textbf{110} (1979), no.~1, 127--142.
  \MR{541332}

\bibitem[Wil74]{Wi74}
R.~F. Williams, \emph{Expanding attractors}, Inst. Hautes {\'E}tudes Sci. Publ.
  Math. (1974), no.~43, 169--203. \MR{0348794}

\end{thebibliography}
\bibliographystyle{amsalpha}
\end{document}